\documentclass[11pt]{amsart}
\usepackage{inputenc}
\usepackage[hmarginratio={1:1},
  vmarginratio={1:1}, vmargin = 2.5cm, hmargin = 2cm]{geometry}

\usepackage{graphicx} % Required for inserting images
\usepackage{subcaption}
\usepackage[export]{adjustbox}
\usepackage{amsmath}
\usepackage{amssymb}
\usepackage{mathrsfs}
\usepackage{xcolor}
\usepackage{algorithm}
\usepackage{algpseudocode}
\usepackage{svg}
\usepackage{tabularx}
\usepackage{multicol}
\usepackage{epstopdf}
\usepackage{url}

\newtheorem{theorem}{Theorem}[section]

\newtheorem{definition}[theorem]{Definition}
\newtheorem{notation}[theorem]{Notation}

\newtheorem{remark}{Remark}

\numberwithin{equation}{section}
\newcommand{\Z}{\mathbb Z}
\newcommand{\Q}{\mathbb Q}
\newcommand{\C}{\mathbb C}
\newcommand{\R}{\mathbb R}

\DeclareMathOperator*{\argmax}{arg\,max}
\DeclareMathOperator*{\argsort}{arg\,sort}

\author[P. Engel]{Peter Engel}
\author[O. Hammond-Lee]{Owen Hammond-Lee}
\author[Y. Su]{Yiheng Su}
\author[D.~Varga]{Dániel Varga}
\author[P.~Zsámboki]{Pál Zsámboki}\thanks{Corresponding author: Pál Zsámboki, e-mail address: \texttt{zsamboki.pal@renyi.hu}\\[3 pt]}

\title{Diverse beam search to find densest-known planar unit distance graphs}
\begin{document}

\maketitle

\begin{abstract}
    
This paper addresses the problem of determining the maximum number of edges in a unit distance graph (UDG) of $n$ vertices using computer search. An unsolved problem of Paul Erdős asks the maximum number of edges $u(n)$ a UDG of $n$ vertices can have. Those UDGs that attain $u(n)$ are called ``maximally dense.'' In this paper, we seek to demonstrate a computer algorithm to generate dense UDGs for vertex counts up to at least 100. Via beam search with an added visitation metric, our algorithm finds all known maximally dense UDGs up to isomorphism at the push of a button. In addition, for $15 < n$, where $u(n)$ is unknown, i) the algorithm finds all previously published densest UDGs up to isomorphism for $15 < n \le 30$, and ii) the rate of growth of $u(n)/n$ remains similar for $30 < n$. The code and database of over 60 million UDGs found by our algorithm can be found at \url{https://codeberg.org/zsamboki/dbs-udg}.
\end{abstract}

\section{Introduction}
A simple graph is a unit distance graph (UDG) if it has an embedding formed by placing distinct points in a Euclidean space $\mathbb R^k$ and connecting any two points with an edge if and only if the distance between them is exactly one. In this paper, we are only concerned with planar UDGs, that is those graphs that have such embeddings to $\mathbb R^2$. We will use the term UDG to refer to both the graph itself and any such embeddings. An unsolved problem of Paul Erdős \cite{erdos1946} asks the maximum number of edges a unit distance graph of $n$ vertices can have, this value denoted as $u(n)$. Those UDGs that attain the maximum number of edges are called ``maximally dense.'' The current best upper bound of $u(n)\le\sqrt[3]{\frac{29 n^4}{4}}$ was established in Ágoston, Pálvölgyi \cite{agoston2022improved}. As for lower bounds, Erdős himself showed that there exists $c>0$ and infinitely many $n$ such that $n^{1+c / \log \log n}\le u(n)$ \cite{erdos1990}.

The value of $u(n)$ is known for $n \leq 15$. For $n$ between 3 and 8, a maximally dense UDG can be constructed by pasting equilateral triangles together. For $9\le n\le 15$, constructions involve approaches such as Cartesian products (for example with the densest known graph at $n = 9$, which is the Hamming graph $H(2,3)$) or rotating equilateral triangles and squares. A prior attempt to find dense UDGs for low vertex counts which employs many of these techniques is found in Schade \cite{schade}.

In this paper, we seek to demonstrate an algorithm to generate dense UDGs. In section 2, we introduce the structural properties of the UDGs that our algorithm finds and the space in which they are embedded. Then in section 3, we introduce the diverse beam search \cite{Vijayakumar} approach that we employ. Section 4 covers the implementation of the algorithm and certain optimizations. Following, in section 5, we show and discuss our results, both a comparison of our findings versus the known bounds and certain UDGs of interest. Finally, in section 6, we discuss future thoughts: potential optimizations, different approaches, and applications of our results.
%Peter: I will rewrite this last paragraph after we finalize our later sections.
%conjecture: for every maximal number, there is an isomorphism embedded in the moser lattice - probably not true, but what can we limit

\section{The Search Space}
Our objects of interest are UDGs in the Euclidean plane. We can view the vertices of these graphs as complex numbers. In this section, we introduce the Moser lattice and construct UDGs on this lattice. Notably, highly dense unit distance graphs often fall on the Moser lattice, and each maximally dense graph presented in Ágoston, Pálvölgyi \cite{agoston2022improved} has been shown to be embeddable on this lattice.

\subsection{The Moser Lattice}
\begin{notation}
    For a positive integer $t\in\mathbb Z_{>0}$, we let
    $$
    \omega_t=\exp(i\cdot\arccos(1 - 1 / 2t)).
    $$
\end{notation}
\begin{definition}
The Moser lattice, denoted as $M_L$, is defined as the additive subgroup
\begin{align*}
    M_L = \{a\cdot 1 + b\cdot\omega_1 + c\cdot\omega_3 + d\cdot\omega_1\omega_3 \mid a, b, c, d \in \Z\} = \Z\langle1, \omega_1, \omega_3, \omega_1\omega_3\rangle\le(\C,+),
\end{align*}
\end{definition}
%Need picture
%of what?
\begin{theorem}
The following properties hold for the Moser lattice $M_{L}$:
\begin{enumerate}
    \item The degree of $M_{L}$ is 4.
    
    \item $\{1, \omega_1, \omega_3, \omega_1\omega_3\}$ is a basis for $M_{L}$.

    \item $M_{L}$ is isomorphic to $\mathbb{Z} \times \mathbb{Z} \times \mathbb{Z} \times \mathbb{Z} = \mathbb{Z}^4$.
\end{enumerate}
\end{theorem}

\begin{proof}
Firstly, consider the minimal polynomials of $\omega_1$ and $\omega_{3}$ over $\mathbb{Q}$. As we have $\omega_1^2 - \omega_1 + 1 = 0$ and $\omega_{3}^2 - \frac{5}{3} \omega_{3} + 1 = 0$, the two minimal polynomials are quadratic and coprime, which implies that $[\mathbb{Q}[\omega_1,\omega_{3}]:\mathbb{Q}] = 4$. That is, the degree of the extension is 4.

Secondly, observe that the 4-element set $\{1, \omega_1, \omega_3, \omega_1\omega_3\}$ generates the $\Q$-vector space $\Q[\omega_1,\omega_3]$. Thus, this set is a basis for $M_{L}$. 

Lastly, define the map $\varphi: \mathbb{Z}^4 \rightarrow M_{L}$, where $\varphi((a,b,c,d)) = a\cdot 1 + b\cdot\omega_1 + c\cdot\omega_3 + d\cdot\omega_1\omega_3$. We show that this map is an isomorphism. The injectivity of $\varphi$ is clear from the linear independence of $1$, $\omega_1$, $\omega_3$, and $\omega_1\omega_3$. The surjectivity of $\varphi$ comes from the definition of $M_{L}$. Thus, $\Z^4 \cong M_{L}$.
\end{proof}
\begin{definition}[Matrix Form of UDG]
    \label{def:matrix form of udg}
    Define the map $\varphi: \mathbb{Z}^4 \rightarrow M_{L}$, where $\varphi((a,b,c,d)) = a\cdot 1 + b\cdot\omega_1 + c\cdot\omega_3 + d\cdot\omega_1\omega_3$. For a UDG $U = (V, E)$ with vertices on $M_L$, the matrix form of $U$ is represented as 
    % \begin{align*}
    %     \begin{bmatrix} \varphi^{-1}(v_1) \\ \varphi^{-1}(v_2) \\\vdots \\ \varphi^{-1}(v_n) \end{bmatrix}
    % \end{align*}
    $U=[\varphi^{-1}(v_1), \varphi^{-1}(v_2), \ldots, \varphi^{-1}(v_n)]$, where the vertex set $V = \{v_1, v_2, \ldots, v_n\}$ and the vertices are ordered.
\end{definition}
The matrix representation for each UDG comes from the inverse mapping of $\varphi$ applied to each vertex in the vertex set $V$. For a UDG with $n$ vertices, we can represent the graph $U$ by an $n \times 4$ integer matrix. This representation helps transform operations on a UDG into equivalent operations on an integer matrix, thereby simplifying computations. An example of such a representation is shown in Figure~\ref{moser}.
\begin{figure}[h]
    \centering
    \begin{minipage}{0.5\textwidth}
        \centering
        $\begin{bmatrix}
            0 & 0 & 0 & 0 \\
            1 & 0 & 0 & 0 \\
            0 & 1 & 0 & 0 \\
            0 & 0 & 1 & 0 \\
            0 & 0 & 0 & 1 \\
            1 & 1 & 0 & 0 \\
            0 & 0 & 1 & 1
        \end{bmatrix}$
    \end{minipage}%
    \begin{minipage}{0.5\textwidth}
        \centering
        \includegraphics[width = 250pt]{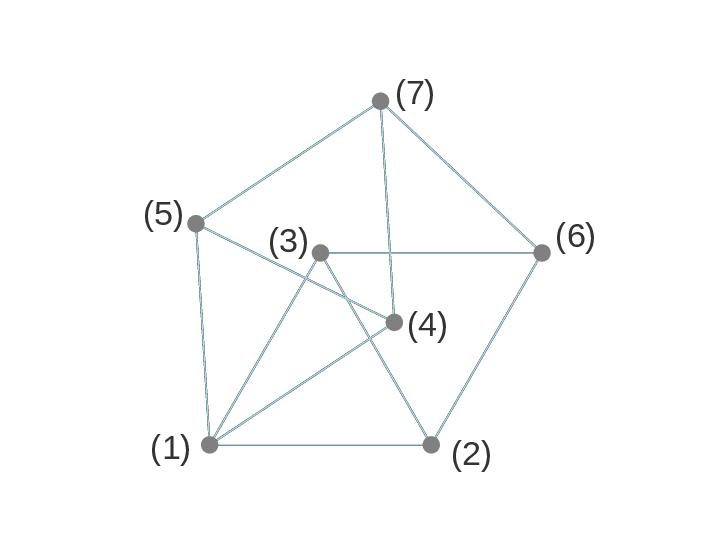}
    \end{minipage}
    \caption{The Moser spindle and its matrix representation.}
    \label{moser}
\end{figure}
% \color{red} Do we want to give an example here? \color{red}
% we could maybe do something like SpindleMat.png

\begin{theorem}\label{thm:18 unit vectors}
In the Moser lattice, there are 18 unit vectors in total.
\end{theorem}

\begin{proof}
Take an arbitrary element $z=a + b\omega_1 + c\omega_3 + d\omega_1\omega_3\in M_L$ with $a,b,c,d\in\mathbb Z$. It is of unit length if and only if
\begin{align*}
    1 &= z\overline{z} \\
      &= (a + b\omega_1 + c\omega_3 + d\omega_1\omega_3)\overline{(a + b\omega_1 + c\omega_3 + d\omega_1\omega_3)} \\
      &= \left(a^2 + ab + \frac53ac + \frac56ad + b^2 + \frac56bc + \frac53bd + c^2 + cd + d^2\right)
        + (bc - ad)\frac{\sqrt{33}}{6}.
\end{align*}
That is, $z$ has length 1 if and only if we have
$$
p(a, b, c, d) := a^2 + ab + \frac53ac + \frac56ad + b^2 + \frac56bc + \frac53bd + c^2 + cd + d^2=1\text{ and }ad = bc
$$

Let's denote by $A$ the symmetric matrix corresponding to the quadratic form $p(a, b, c, d)$:
$$
\text{with }\mathbf v:=\begin{pmatrix} a \\ b \\ c \\ d\end{pmatrix}\text{ and }
A := \begin{pmatrix}
    1 & \frac12 & \frac56 & \frac {5}{12} \\
    \frac12 & 1 & \frac{5}{12} & \frac56 \\
    \frac56 & \frac{5}{12} & 1 & \frac12 \\
    \frac{5}{12} & \frac56 & \frac12 & 1
\end{pmatrix}\text{ we have }p(a, b, c, d)=\mathbf v^TA\mathbf v.
$$
Let's take orthogonal eigenvectors:
$$
\text{with }Q:=\frac12\begin{pmatrix}
    1 & 1 & -1 & -1 \\
    1 & -1 & 1 & -1 \\
    1 & 1 & 1 & 1 \\
    1 & -1 & -1 & 1
\end{pmatrix}\text{ and }D:=\begin{pmatrix}
    \frac{11}{4} & 0 & 0 & 0 \\ 0 & \frac{11}{12} & 0 & 0 \\ 0 & 0 &  \frac{1}{12} & 0 \\ 0 & 0 & 0 & \frac14
\end{pmatrix}\text{ we have }A = QDQ^T.
$$
In particular, for $\begin{pmatrix}
    \tilde a & \tilde b & \tilde c & \tilde d
\end{pmatrix}^T := \tilde{\mathbf v} := Q^T\mathbf v$, we get 
$$
1 = \mathbf v^TA\mathbf v = \mathbf v^T QDQ^T\mathbf v^T = \tilde{\mathbf v}^TD\tilde{\mathbf v}
 = \frac{11}{4}\tilde a^2 + \frac{11}{12} \tilde b^2 + \frac{1}{12} \tilde c^2 + \frac14\tilde d^2
      \ge \frac{1}{12}\left(\tilde a^2 + \tilde b^2 + \tilde c^2 + \tilde d^2\right)
$$
Now as $Q$ is orthonormal, from $\|\tilde{\mathbf v}\|\le4$ we get $\|\mathbf v\|\le4$ and thus $-4\le a, b, c, d\le 4$. The finitely many choices for $a,b,c,d\in\mathbb Z$ we get can be checked programmatically to yield that indeed there are exactly 18 possibilities, displayed in Figure~\ref{units}.
\end{proof} %units.svg should we add labels?
Note that by this result, the degree of any vertex in a graph on the Moser lattice is at most 18. This limits the asymptotic behavior of maximally dense graphs on the lattice to below the previously established bounds. However, for relatively small graphs the Moser lattice is still a useful tool for studying maximally dense UDGs, as demonstrated by the fact that our search procedure finds embedded versions of all the graphs presented in Ágoston, Pálvölgyi \cite{agoston2022improved}, the collection of previously known densest graphs.

\begin{remark} (1) Inspired by the first preprint of this work, it was proven that in the family of lattices
    $$
    L_k=\{a\cdot1 + b\cdot\omega_1 + c\cdot\omega_k + d\cdot\omega_1\omega_k\mid a,b,c,d\in\mathbb Z\},
    $$
    the maximum number of unit vectors is unbounded \cite[Theorem 2.1]{ruhland2024families}.

    (2) Let $\alpha$ be a complex root of the polynomial $p(z)=z^4-z^3-z^2-z+1$.
    Then the lattice
    $$
    L=\{a\cdot1 + b\cdot\alpha + c\cdot\alpha^2 + d\cdot\alpha^3\mid a,b,c,d\in\mathbb Z\}
    $$
    has infinitely many unit vectors \cite[Theorem 2]{radchenko2021unit}.

    We checked some of these alternative lattices, but none gave better results than the Moser lattice.
\end{remark}

\subsection{Canonization} \label{sec:canon}
Consider some matrix $[[p_1], [p_2], ..., [p_n]]$ where each $[p_i]$ is a point on the Moser lattice, $[p_i]=[a_i, b_i, c_i, d_i] = a_i + b_i \cdot \omega_1 + c_i \cdot \omega_3+ d_i \cdot \omega_1 \omega_3$. Together, this array of points will constitute the vertex set of a planar unit distance graph. Immediately, there are $n!$ identical copies of this graph according to the permutations of this array. This may be addressed by treating this list as a set object.\\
However, when considering rotations and reflections, there are other possible embeddings of this graph in the Moser lattice. When considering translation as well, there are infinitely many more. Theoretically, this poses no issues. A search algorithm would still progress. In practice, however, treating multiple embeddings of the exact same graph increases run time for no added benefit.
%"treating" doesn't seem to be the right word
%Say how much? --- We can test removing the algorithm but we would need a bench mark to reach - we can test it and see what it fails on, pick the last graph it gets in reasonable time, and compare? Or pick beforehand for a better comparison? In general we know 12 graphs would be visited from rotations and reflections, but translation is harder to quantify empirically.

The solution to this is to find a formula for creating a canonical representation of a graph in order to map these identical graphs to this canonical one.

In the Moser lattice, we find $12$ representations of a graph through rotations by $\frac{\pi}{3}$ radians and reflections swapping the pair of generators $\omega_1$ and $\omega_3$ along with the pair $\omega_1 \omega_3$ and $1$.

 To show why a $\frac{\pi}{3}$ radians rotation is guaranteed to remain on the Moser lattice, consider an alternative representation of the first three generators as follows: $$1 = e^0, \omega_1 = e^\frac{\pi i}{3}, \omega_3 = e^{i \cdot 2 \arcsin(\frac{1}{\sqrt{12}})}$$ Thus, the $\frac{\pi}{3}$ rotation is equivalent to the following mapping of generators: $$a \mapsto a \cdot \omega_1, b \cdot \omega_1 \mapsto -b + b \cdot \omega_1, c \cdot \omega_3 \mapsto c \cdot \omega_1 \omega_3, \text{ and } d \cdot \omega_1 \omega_3 \mapsto -d \cdot \omega_3 + d \cdot \omega_1 \omega_3.$$ 
 
Since we have defined the reflection and the rotation in terms of linear transformations of the generators, they may be interpreted as the following matrices: $$ \text{Rotation} = \begin{bmatrix} 0 & 1 & 0 & 0 \\ -1 & 1 & 0 & 0 \\ 0 & 0 & 0 & 1 \\ 0 & 0 & -1 & 1 \end{bmatrix} \text{; Reflection} = \begin{bmatrix} 0 & 0 & 0 & 1 \\ 0 & 0 & 1 & 0 \\ 0 & 1 & 0 & 0 \\ 1 & 0 & 0 & 0 \end{bmatrix}$$
%Owen: These matrices still hold for 18 but we should prove why. Also note there are still only 12 representations as the 18 do not afaik correspond to more possible orientations. 
Thus, for an $n \times 4$ matrix, where rows represent points and columns represent generators, representing a UDG in $M_L$, finding the 12 representations is a function $F$ from the space of integer valued $n \times 4$ matrices to the space of integer valued $n \times 4 \times 12$ tensors according to $U \mapsto [U, URo, U Ro^2, URo^3, URo^4,URo^5,$ \\$ URe, UReRo , UReRo^2, UReRo^3, UReRo^4, UReRo^5]$. For a collection of $m$ UDGs, extend this to be a function from $\mathbb{Z}^{m \times n \times 4}$ to $\mathbb{Z}^{m \times n \times 4 \times 12}$ such that $[U_1, ..., U_m] \mapsto [F(U_1), ..., F(U_m)]$. 

%Now, one should note that these are not all of the possible translation-invariant embeddings of a graph in the Moser lattice. This is a choice that we have made, and it is elaborated on in the Miscellaneous section. A more thorough canonization algorithm would catch more equivalent graphs, but we believe that the additional computational cost would outweigh any improvement. This is also supported by the fact that we have not observed equivalent pairs of UDGs in our top results.\\
%Owen: Is this still true? We have decided for graphs containing generators this should catch all invariant graphs within the translation range correct?
%Peter: I will write something about this. I've commented out the old paragraph.
Now, the question of translation. After rotations and reflections, a UDG has 12 representations of itself, each as an $n \times 4$ integer matrix. For each of these 12, it is only necessary to perform a uniform translation of the generators---those being the columns of the matrix---such that the minimal coefficient of each generator in the matrix is 0. % Will explain more. But should I leave the selection to the Zobrist hash

Given this final tensor in $\mathbb{Z}^{m \times n \times 4 \times 12}$, the question remains of how to map it to $\mathbb{Z}^{m \times n \times 4}$ by selecting one of the $12$ $n \times 4$ matrices for each of the $m$ graphs with the condition that if another tensor differs only by permutations of the $n \times 4$ matrices in each of the $m$ rows, the image will be the same. In other words, a canonical method of selecting one of the 12 representations. To do this, we use Zobrist hashing, which will be covered in section ~\ref{sec:zzz}, with the entire algorithm shown in ~\ref{sec:cana}.

Note that it is possible to construct two equivalent graphs that have distinct canonical forms. For example, $t_ 1 = \begin{bmatrix} 1 & 0 &  0 & 0 \\ 0 & 1 & 0 & 0 \\ -1 & 1 & 0 & 0 \end{bmatrix}$ and $t_2 = \begin{bmatrix} 2 & -1 & -2 & 1 \\ 1 & 1 & -1 & -1 \\ -1 & 2 & 1 & -2 \end{bmatrix}$ represent two triangles on the intersection of the unit circle and the Moser lattice with identical angles. However, the canonization algorithm maps $t_1 \mapsto t_1' = \begin{bmatrix} 0 & 0 & 0 & 0 \\ 0 & 0 & 1 & 0 \\ 0 & 0 & 1 & 1 \end{bmatrix}$ and $t_2 \mapsto t_2' = \begin{bmatrix} 1 & 3 & 0 & 0 \\ 0 & 2 & 1 & 1 \\ 1 & 0 & 0 & 3 \end{bmatrix}$.

Still, this canonization algorithm is a very useful tool as it vastly reduces the number of UDGs we need to go through in our search, and we can make a very fast implementation of it.

\begin{minipage}{0.35\textwidth}
\begin{figure}[H]
$\begin{bmatrix}
    -2&1&2&-1\\
-1&-1&1&1\\
-1&0&0&0\\
-1&1&0&0\\
-1&2&1&-2\\
0&-1&0&0\\
0&0&-1&0\\
0&0&-1&1\\
0&0&0&-1\\
0&0&0&1\\
0&0&1&-1\\
0&0&1&0\\
0&1&0&0\\
1&-2&-1&2\\
1&-1&0&0\\
1&0&0&0\\
1&1&-1&-1\\
2&-1&-2&1\\
\end{bmatrix}$
\end{figure}
\end{minipage} \hfill
\begin{minipage}{0.65\textwidth}
\begin{figure}[H]
\includegraphics{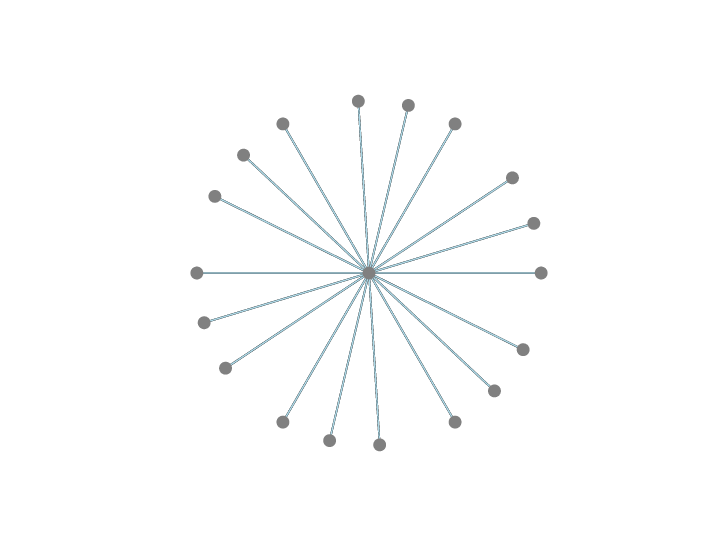}
\caption{The units of the Moser lattice with connections to the origin.}
\label{units}
\end{figure}
\end{minipage}

\section{Diverse Backtracking Beam Search}

In this study, we designed a heuristic beam search algorithm, which is a breadth first search algorithm with pruning. The beam search is used to find UDGs with optimal edges. First, we define the concepts of ``children'' and ``parents'' in the context of UDGs. Then, we introduce the forward and backward steps. Backtracking Beam Search is the integration of those two steps. Finally, in Subsection \ref{sec:dei}, we describe the outer loop with visitation count.

\begin{definition}[Children of a UDG] 
Let $U$ be a UDG with $n$ vertices. A child of $U$ is defined as the canonized version of any UDG that contains $U$ as a subgraph and has exactly $n + 1$ vertices.
\end{definition}

\begin{definition}[Parents of a UDG] 
Let $U$ be a UDG with $n$ vertices. A parent of $U$ is defined as the canonized version of any UDG of which $U$ is a subgraph and has exactly $n - 1$ vertices.
\end{definition}

\subsection{Starting Step}
Our beam search starts with the Moser spindle, which we found to give the best results after experimenting with various other starting graphs. During the first round of the search, we generate children from the Moser spindle by applying a predefined set of operations. Then, the top $\alpha$ UDGs are retained based on their edge number. If the number of generated children is less than $\alpha$, all children are preserved for the following search rounds. From here on, the algorithm chooses between forward and backward steps.

\subsection{Forward Step}
A similar procedure is followed for later forward steps: children are generated from the results of the $(n-1)^{\text{th}}$ round using a predefined set of operations. The top $\alpha$ UDGs are selected based on the number of edges. If the number of children produced in this round is fewer than $\alpha$, all children are preserved for the following search rounds. The hyperparameters for the beam search include the number of search iterations, $\alpha$, and the specific set of operations used for constructing children in each search round. The detailed choice of hyperparameters is in section ~\ref{sec:imp}.
% pseudocode?
%I wrote something to this effect further down in implementation, feel free to steal some of those sentiments and the psuedocode can go in the section where the operations are now.
% Yiheng: I redefined the parents and children in general. But, I think it is good to talk about implementation of them and operations in section 5 again.
\subsection{Backward Step}
The backward step does the opposite process of the forward step. In the $n^{th}$ round of the search, instead of finding children from $(n-1)^{th}$ round, the backward step finds all the parents of UDGs by deleting one vertex along with its associate edges. Then, we retain the top $\alpha$ UDGs based on the number of edges. The backward step allows us to potentially discover additional UDGs that may not have been identified in the previous search.
%We have a new hyperparameter $\beta$ in the backward beam search.
The pseudocode for our full diverse backtracking beam search implementation is included in section ~\ref{sec:backward}.
% pseudocode?
%update backwards image and get it

\subsection{Motivations}
%where do we think this section should go - probably after parent/children algorithm, and have a brief paragraph at the beginning of section 4 to provide motivation? it seems difficult to parse backwards beam before understanding the operations.
%pde: I honestly think it's fine here. It's not especially important to know exactly how they are generated to understand why a backward beam search is desired.
A traditional beam search for the UDG problem is quite easy to motivate. Suppose we wanted to find every connected graph possible to compare their edge counts, which would indeed find all edge-dense graphs. One way to do so is to start with a single vertex, and then generate every possible unique 2-vertex graph under canonization. Note this reduces to adding a single vertex in unique locations relative to the current graph. Then we repeat this process on every 2-vertex graph found to generate graphs on three vertices, then four, etc. While this process would find all graphs eventually, it is simply not feasible. There are finitely many choices at each step (according to the 18 unit vectors), but the total number of graphs seen in the BFS would be too large.
So to reduce the time spent at each step, a beam search takes this basic model and cuts the number of children graphs used to generate the next group of graphs to a certain number, and we denote this final selection a beam. The remaining details covered in this section are then (1) How to generate new graphs efficiently (2) How to choose the best graphs for the beam (3) How to otherwise improve upon beam search.
%This section needs rewriting, but waiting until after we pick section order to do so
%Owen: if this section is fine here I will review and rewrite and ask for a peer review
%ToDo: shorten this section above at the beginning of the section

To motivate the backward step, note that edge-dense graphs are often children or parents of each other, as since each additional vertex can only add so many edges, the previous edge count must have already been high. Then if we find an edge-dense graph in beam search, not only will adding vertices to it result in edge-dense graphs, but removing vertices will as well. However, traditional beam search does not allow for finding smaller relatives of edge-dense graphs. Backtracking Beam Search solves this issue by traversing up and down in graph size as long as edge-dense graphs are still being found.
Begin as in beam search by getting the children of the current beam, and take the beam of these children. If any of these children are new in this iteration and maximal among found graphs, we get the parents of this child beam. Then we check these parent graphs for new and maximal graphs, and continue taking the parent beam until we find none. We then take the final parent beam and take the children, assuming we find no new and maximal graphs we add these children to the parent set of the same size and take the beam of this combined collection, then continue getting children and combining the children and parent sets until we finally generate a new children set and combine it with the first one generated. If at any point we encounter more new and maximal graphs we once again begin this process.
%May be better to write this in english but with looping and in steps to illustrate. Just needed to draft it.
%Owen: thoughts on the above? depends on whether we write or cite for backwards - I don't think this paragraph adequately explains Backward Beam, if we need a paragraph form explanation I will rewrite to work with psuedocode to explain.
%ToDo: psuedocode instead
\subsection{Diversity (visitation)} \label{sec:dei}
Beam search on its own is a greedy algorithm, as at each step of pruning, the ordering is based on the edge count at the current level. This works on the assumption that graphs with maximal edge counts at higher vertex counts have close to maximal edge counts at lower vertex counts. However, this does not always hold, and there are examples of maximally dense UDGs with significantly sub-optimal ancestors.
%Maybe I talk about 27 graph?
A solution to this issue is to implement a measure of diversity \cite{Vijayakumar}. Combining this number with the edge count to determine the ordering of graphs in pruning encourages the algorithm to better explore the search tree.

Our solution to this is a visitation count penalty. Naturally, this means that for the algorithm to be effective, the beam search is to be run multiple times. In each run, if a UDG is seen (only counting those that are kept in pruning), its visitation count is incremented. Then in the following runs, its score in the sorting for the pruning is decremented by its visitation count. Over the course of many runs, this allows for the discovery of UDGs with sub-optimal ancestors while still prioritizing the consistently dense sub-trees.
%TODO: add sentence about other uses of visitation
\section{Implementation}\label{sec:imp}
%zobrist hash, visitation array using zobrist heads
%-vectorization\\
In order to successfully run the algorithms described above, several lower level choices must be made in how the operations required are processed by the computer. In addition, certain operations such as canonization require sub-processes independent of the larger algorithm, yet important in its performance nonetheless. Such choices are described in this section in more detail. 
\subsection{Vectorization}
In theory, given enough computational resources, even a simple tree search could eventually find every isomorphism class of connected UDGs in the Moser lattice up to a given vertex count. However, with limited resources and time, optimizing the implementation of the beam search is essential to efficiently produce results for larger numbers of vertices. In particular, given the enormous number of possible graphs generated by the beam search, the cost associated with performing operations such as canonization, finding children and parents, and evaluating edge counts for individual graphs is incredibly inefficient. Instead, we aim to vectorize all of these operations such that given a collection of graphs all of size $n$, we perform the operation simultaneously on all the graphs. This approach also allows the implementation to use GPU processing, which is much more efficient for such array manipulation. In specific our implementation uses vectorized array operations, representing the collection of graphs as three dimensional array of Moser lattice coefficients, each a small integer.
%Owen: List of vectorized operations here? or note in each of the below sections which operations are vectorized. And do we want to go deeper into explanation of GPU, and/or introduce it in the intro of the implementation section?
\subsection{Chunking}
In addition to optimizing functions for GPU processing, memory constraints also posed an important obstacle to efficiency. In particular the number of children returned at each size $n$ is roughly $n^4$ times the number of parents (determined by the first child operation), and even though an array storing all of these results could fit in memory, the matrices resulting from such an array passing through the canonization algorithm would not. So we manually set bounds for the maximum lengths of arrays that could be passed to the various functions at once. If an array was longer than the limit, it would be split into ``chunks'' of length equal to the limit. In particular, the limit scaled inverse to the number of points in the graph, as more points means more memory required for each graph. The only arrays invariably stored in memory at once were the visitation array and dictionary of new graphs.
\subsection{Children}
The implementations for finding children and parents are heavily optimized for the specific memory and processing constraints we had. To that end, in our implementation, for an array of parent graphs, that array is passed to each of three operation functions, which return an array of children corresponding to the three operations from section 3.1. Those results are concatenated and canonized to return only unique values. The three operation functions are as follows:
\subsubsection{Operation 1: Single Edge Addition}
Given a parent graph $G$ and a vertex $v \in V(G)$, generate $V(G') = V(G) \cup v'$ for a $v'$ in the set of generators times $\pm 1$, $\{\pm 1, \pm \omega_1, \pm \omega_3, \pm \omega_1 \omega_3 \}$. \\
Then to generate the full list of children under this operation, the 8 possible selections are represented as an $8 \times 4$ integer matrix. Given a current beam of $m$ UDGs represented as $n \times 4$ matrices, a vectorized addition is performed to get all of the $m \cdot n \cdot 8$ candidate new vertices. 
%Given a parent graph $G$ and vertex $v\in V(G)$, generate $V(G')=V(G) \cup v'$ for $v'$ a unit distance from $v$, $v'\notin G$.\\
%Then to generate the full list of children for $G$, we take every possible unit vector $u$ in the Moser lattice and every vertex in $G$, generate every possible combination of new vertices $v'=v+u$, and add the set of distinct and new vertices to $G$.\\
%In particular, the set of unit vectors can be represented as a matrix of coefficients, and using this representation these operations may be performed vectorized over a collection of $G$ using matrix operations. The matrix of units and their representation in the Moser lattice with connections to the origin are shown in Figure~\ref{units}.\\

%additional 4-6 points here from testing, future implementations should include all
%Owen: We should write about this in the future thoughts section, or include all 12 but with a footnote that only 8 were known when the database was developed.
%figure showing an example
\subsubsection{Operation 2: Triangle Completion}
Given a parent graph $G$ and pair of vertices $u,v\in V(G)$ a unit distance apart, generate $V(G')=V(G)+v'$ for $v'$ a unit distance from both $u$ and $v$.\\
The formula to generate such a $v'$ as a function of the coefficients of $v$ and $u$ as vectors in $\mathbb{Z}^4$ is given by $v'=u+(v-u)r$, where $r$ is the $\frac{\pi}{6}$ rotation matrix. Note that taken $v$ as the first vertex generates the other possible triangle completion of $u,v$.
%figure showing example
%psuedocode, explain in brief the vectorized algorithm.\\
\subsubsection{Operation 3: Parallelogram Completion}
Given parent graph $G$ and vertices $v,u,w$ such that pairs $v,u$ and $v,w$ are unit distance apart generate $V(G')=V(G)+v'$ for $v'$ such that pairs $v',u$ and $v',w$ are a unit distance apart.

%figure for example of operation
\subsection{Parents}
Finding clusters of edge-dense graphs using backtracking beam search relies also on finding subgraphs which lead to larger maximal graphs. These graphs can also therefore be found as parents of previously discovered graphs in the search. The algorithm for discovering parents is relatively simple and low-cost, and is as follows.\\
%I don't know if we need to include a paragraph to this effect, it depends on how the back-beam search section looks when it is done.
Given a graph $G$, we obtain the parent set of $G$ by creating some $V(G')=V(G)-v$ for all $v\in V(G)$. The set of unique canonized parents of $G$ is thus at most the size of $G$, and some subgraphs may reduce to the same canonized graph.

%-before vectorization introduce implementation\\
%-Introduce notion of GPU computation/vectorization\\
%-higher dimensional array operations\\
%-specific implementations?\\
%List of algorithms vectorized:\\
%%-canon\\
%-children (operations individually)\\
%-parents\\
%-eval\\
%-adj matrix (needed for children)\\
\subsection{Adjacent Edge Matrix}
In order to complete children operations 2 and 3, as well as evaluate the final edge counts of graphs, the set of edges in the unit-distance graph is required. To do so, the coordinates are multiplied by the generating set of the Moser lattice to produce a length-$n$ array of coordinates in the complex plane, where $n$ is the number of vertices in the graph. A new matrix is produced using slicing operations with the difference of each distinct pair of points located at the row and column index of those points. The absolute value of this matrix is taken, and then each value is checked to within a tolerance of 1, evaluating to 1 if so and 0 otherwise. A tolerance of $1\times10^{-8}$ was used, and is successful because of the limit in coefficient values to $-10$ through 10.
%Do you think we should cite something here? -- Peter
\subsection{Zobrist Hash}\label{sec:zzz}
The Zobrist Hash \cite{zobrist}, used most commonly in uniquely identifying chess board positions, is used here to give a permutation-invariant identification for each graph. That is, for a collection of points in our Moser lattice, the Zobrist hash will be identical no matter in which order those points are presented in the array. To produce the Zobrist hash, we will assign a different random 64-bit integer to each possible combination of values the four coefficients that represent a point on the Moser lattice can take. To this end, we must restrict the range of the coefficients---for the scale of the search presented here, [-10,+10] as lower and upper bounds sufficed. Then, we initiate a size $21^4=194481$ random array of 64-bit integers, and given a graph, the Zobrist hash is the exclusive-or of the random integers corresponding to the points in the graph. See Figure \ref{zobt} for an example.
The Zobrist hash as the final form of the canonization algorithm allows for permutation-invariance, but the Zobrist hashes alone cannot be transformed back into the original graph---instead, we created a visitation array.
%put formula for zobrist hash collisions here,  and note we do not care about a small number of collisions that much
%what xor is
%Note that this operation can be vectorized, which is why it was chosen
%The zobrist hash as the final form of the canonization algorithm allows for permutation-invariance, but the zobrist hashes alone cannot be transformed back into the original graph - instead we created a dictionary [with the first permutation of each graph becoming the canonical permutation]
%ToDo: prettier graphic

\begin{figure}
\begin{multicols}{3}
\begin{tabular}{|m{3em}|m{5em}|}
\hline
0 & 11011101 \\
\hline
1 & 01100001\\
\hline
2 & 00000101\\
\hline
3 & 01011011\\
\hline
4 & 10101110\\
\hline
5 & 10110000\\
\hline
6 & 00000111\\
\hline
7 & 01101011\\
\hline
8 & 11101000\\
\hline
\end{tabular}

\begin{tabular}{|m{3em}|m{5em}|}
\hline
0 & 11011101 \\
\hline
\textbf{1} & \textbf{01100001}\\
\hline
2 & 00000101\\
\hline
\textbf{3} & \textbf{01011011}\\
\hline
4 & 10101110\\
\hline
5 & 10110000\\
\hline
6 & 00000111\\
\hline
7 & 01101011\\
\hline
\textbf{8} & \textbf{11101000}\\
\hline
\end{tabular}

\begin{tabular}{|m{3em}|m{5em}|}
\hline
\textbf{XOR} & \textbf{00010010}\\
\hline
\end{tabular}
\begin{figure}[H]
\includegraphics[width=.3\textwidth]{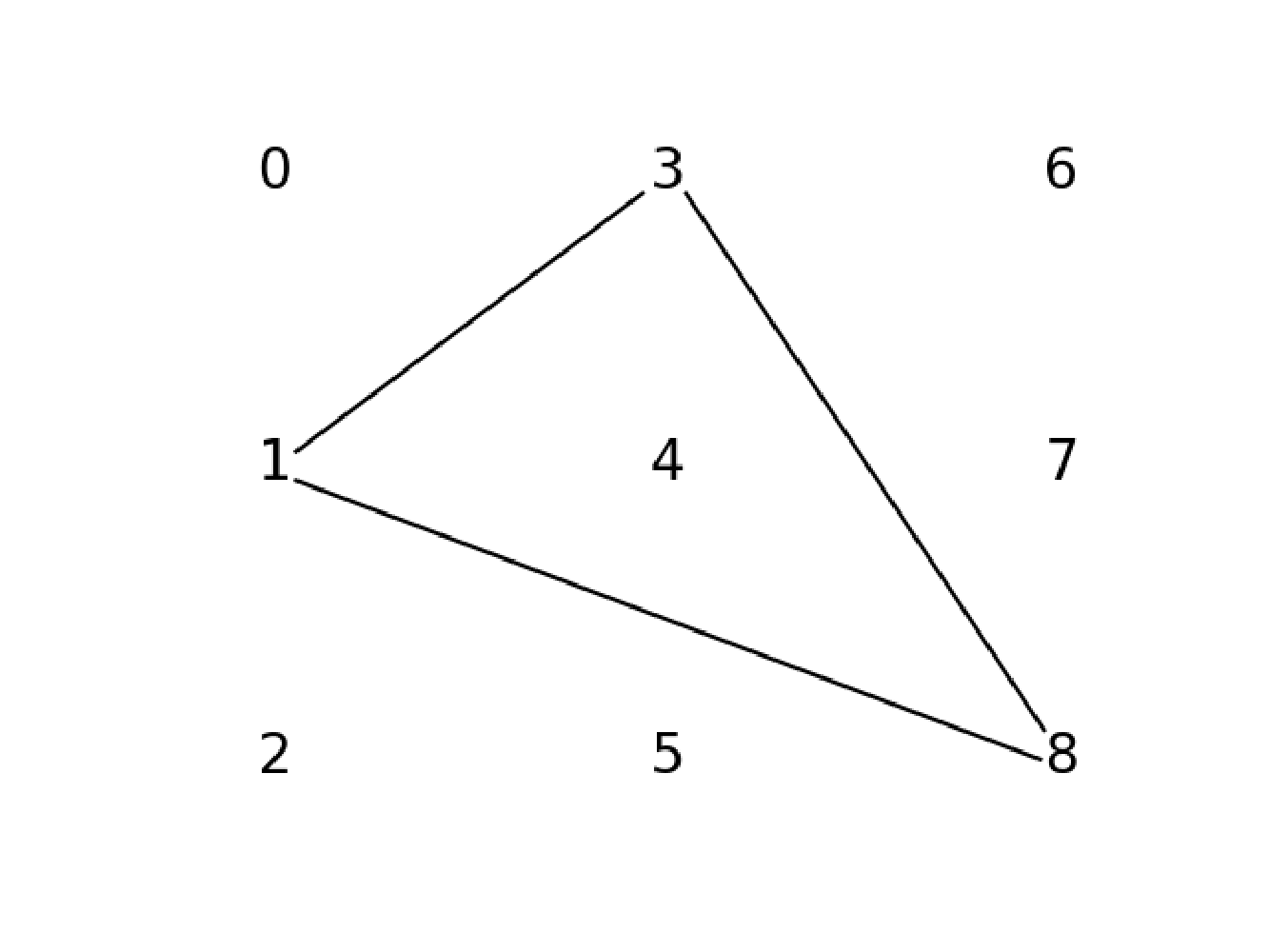}
\label{zobrist}
\end{figure}
\end{multicols}
\caption{The Zobrist Hash}
\label{zobt}
\end{figure}
\subsection{Visitation Array}
Once again the Zobrist hashes come into use, this time in the implementation of ``visitation,'' a notion introduced in section ~\ref{sec:dei}. To store the visitation as an integer for each graph, we use the Zobrist hashes as indices, and store the visitation count for each graph in an array. Because the hashes are 64 bits, an array with an index for each possible hash would be $2^{64}$, far too large to handle in memory on in storage. Instead we choose some number of bits from the front of each hash (the ``head'') to serve as the index---in our implementation an array of size $2^{28}$ is initiated and the first 28 bits of each hash is the index of its visitation count. While this increases the odds of collisions, the expected number of collisions is still small and unlikely to noticeably affect the discovery of any maximal graphs.

\subsection{Canonization Algorithm} \label{sec:cana}
With Zobrist hashing now explained, the canonization algorithm from section 3.2 is now completed as Algorithm~\ref{alg:canon}.

%It can be difficult to mediate between universality (as in, just writing **pseudo**code) and the actual Python implementation. Let me know if you feel that I erred in either direction anywhere.
\begin{algorithm}
\caption{Canonization\\The argument, $B$, represents $m$ UDGs of $n$ vertices. Note that in our implementation, we let $max\_size$ be 21. Additionally, $Re$ and $Ro$ are as defined in section \ref{sec:canon}.}\label{alg:canon}
\begin{algorithmic}
\State $R \gets [\mathbb{I}_4, Ro, Ro^2, Ro^3, Ro^4, Ro^5, Re, ReRo, ReRo^2, ReRo^3, ReRo^4, ReRo^5]^\top \in \mathbb{Z}^{4 \times 4 \times 12}$
\State $Keys \gets \text{Zobrist hashing keys} \in \mathbb{Z}^{max\_size^4}$
\State $C = [max\_size^i \text{ for } i \in [0, 3] \cap \mathbb{Z} ]$
\Procedure{Canonization}{$B \in \mathbb{Z}^{m \times n \times 4}$}
\State $A_{i, r, j, l} \gets \sum_k B_{i, j, k} \cdot R_{l, k, r}$
\Comment{$A \in \mathbb{Z}^{m \times 12 \times n \times 4}$}
\State $M_{i, r, l} \gets \min_j A_{i, r, j, l}$
\Comment{$M \in \mathbb{Z}^{m \times 12 \times 4}$}
\State $A_{i, r, j, l} \gets A_{i, r, j, l} - M_{i, r, l}$
% In the actual code, we first expand to (m, 12, 1, 4) and then broadcast the subtraction to make the fin result (m, 12, n, 4). This is equivalent and easier to read though
\State $L_{i, r, j} \gets \sum_l A_{i, r, j, l} \cdot C_l$
\Comment{$L \in \mathbb{Z}^{m \times 12 \times n}$}
\State $V_{i, r, j} \gets Keys_{L_{i, r, j}}$
\State $D_{i, r} \gets \bigoplus_j V_{i, r, j}$
\Comment{$\bigoplus$ is a bitwise XOR; $D \in \mathbb{Z}^{m \times 12}$}
\State $I_i \gets \argmax_jD_{i,j}$
\Comment{$I \in \mathbb{Z}^m$}
\State $H_i \gets D_{i, I_i}$
\Comment{$H$ represents the hashes; $H \in \mathbb{Z}^m$}
\State $A_{i, j, l} \gets A_{i, I_i, j, l}$
\Comment{$A \in \mathbb{Z}^{m \times n \times 4}$}
\State $U_{i, j} \gets \sum_l A_{i, j, l} \cdot C_l$ \Comment{$U \in \mathbb{Z}^{m \times n}$}
\State $I \gets \argsort(U)$
\Comment{Performed along the rows}
\State $A_{i, j, l} \gets A_{i, I_{i, j}, l}$
\Comment{$A \in \mathbb{Z}^{m \times n \times 4}$}
\State \textbf{return} A, H
\EndProcedure
\end{algorithmic}
\end{algorithm}

\subsection{Diverse Backtracking Beam Search Algorithm} \label{sec:backward}
In this vectorized setting, the backtracking beam search algorithm explained in 3.2 and developed in Matolcsi et al. \cite{matolcsi2023fractional} is implemented with diversity (visitation) as Algorithm~\ref{alg:back}.
%This is the pseudocode for backward beam search, a long algorithm. It excludes the modifications to kept_udgs and delete_counter, as they are in some sense not part of the algoritm itself but rather information generated by the algorithm that is used elsewhere. If you disagree, I will happily include them.
\begin{algorithm}
\caption{Diverse backtracking beam search. Where $B$ is the UDG beam, $L$ is the indices of the highest edge count UDGs in $B$, $M$ contains the maximum old scores, and $O$ contains those edge counts.}\label{alg:back}
\begin{algorithmic}
\Procedure{Backward}{$B, L, M, O$}
\State $\mathscr{B}, \mathscr{L}, \mathscr{M}, \mathscr{O} \gets \{B\}, \{L\}, \{M\}, \{O\}$ \Comment{Initialize lists with current values}
\State $PB, PL, PM, PO \gets B, L, M O$ \Comment{Initialize parents with current values}
\While {True}
\State $PB, PL, PM, PO \gets \text{Parent values associated with } PB$
\State \textbf{break if} $PB$ is empty \textbf{or} the vertex counts are $\leq 4$
\State $\mathscr{B}, \mathscr{L}, \mathscr{M}, \mathscr{O} \gets \mathscr{B} \cup \{PB\}, \mathscr{L} \cup \{PL\}, \mathscr{M} \cup \{PM\}, \mathscr{O} \cup \{PO\}$ \Comment{Temporally ordered lists}
\State \textbf{break if} $PM_{PL_{{1}_1}} < T_{\dim(PB_1)}$ \Comment{$T$ is the list of highest edge counts at each vertex}
\State $C, H \gets \textsc{Canonization}(PB_{PL})$ \Comment{The indices of $PB$ from $PL$}
\State \textbf{break if} all visitation counts of $C$ are not $0$
\EndWhile
\State \textbf{return} $B, L, M, O$ \textbf{if} $\mathscr{B} \setminus \{PB\} = \emptyset$ \textbf{else continue} \Comment{List of parents has one element}

\For {$\ell \gets \# \mathscr{B} $ \textbf{to} $1$}
\State $B, L, M, O \gets \mathscr{B}_{\ell}, \mathscr{L}_\ell, \mathscr{M}_\ell, \mathscr{O}_\ell$ \Comment{$\ell$th element in list}
\State $CB, CL, CM, CO \gets$ Child values associated with $B$
\State $check \gets$ False
\If {$\# CB = 0 \lor \dim(CB_0) \leq 6$} \textbf{continue}
\ElsIf {$CM_{{CL_1}_1} > M_{{L_1}_1}$} $check \gets$ True
\ElsIf {$CM_{{CL_1}_1} = T_{\dim(CB_1)}$}
\State $C, H \gets \textsc{Canonization}(CB_{CL})$ \Comment{The indices of $CB$ from $CL$}
\State $check \gets$ at least one zero visitation count in $C$
\EndIf
\If {$check$}
\State $GB, GL, GM, GO \gets \textsc{Backward}(CB, CL, CM, CO)$ \Comment{Recursive step}
\State $check \gets \# GB = 0$
\EndIf
\State $AO, AM \gets O \cup CO \cup GO, M \cup CM \cup GM$
\State $SO \gets$ reversed sort of $AO$
\State $t \gets SO_{\min(b_{\dim(B_1)}, \dim(AO_1))}$ \Comment{$b$ is the beam width array; $t$ is the score threshold for viability}
\State $B \gets \{B_i : O_i \geq t\} \cup \{CB_i : CO_i \geq t\} \cup \{GB_i : GO_i \geq t\}$ \Comment{Third set is empty if not $check$}
\State $O, M \gets \{AO_i : AO_i \geq t\}, \{AM_i : AO_i \geq t\}$
\State $\mathscr{B}_\ell, \mathscr{M}_\ell, \mathscr{O}_\ell \gets B, M, O$
\EndFor
\State \textbf{return} $\mathscr{B}_1, \{i : \mathscr{M}_1 = \max(\mathscr{M}_1) \}, \mathscr{M}_1, \mathscr{O}_1$

\EndProcedure 
\end{algorithmic}
\end{algorithm}

\section{Results and Discussion} 
%--Bunch of pretty graphs
%--Graph of our results vs. previously known vs. bounds
% findings
% Another idea is to include here the story of how we were unable to get the 27 vertex graph from the paper with beam search,
% via backtracking found out that the best 15 vertex ancestor of the 27 vertex graph has 5 less edges than the best,
% and finally we found the 27 vertex graph with visitation.
%
% Maybe even include individiual and aggregated statistics on visitation search runs.
In order to help illustrate the problem, Figure~\ref{g1g9} depicts one of the maximal UDGs for vertex counts 1--9.
\begin{figure}[h]
\centering
\setlength{\lineskip}{0pt}

\makebox[0pt][r]{}
\includegraphics[width=0.25\textwidth]{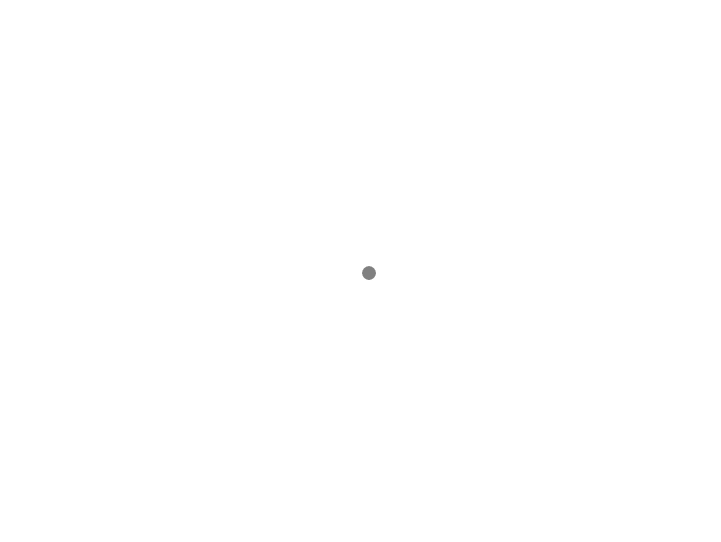}
\includegraphics[width=0.25\textwidth]{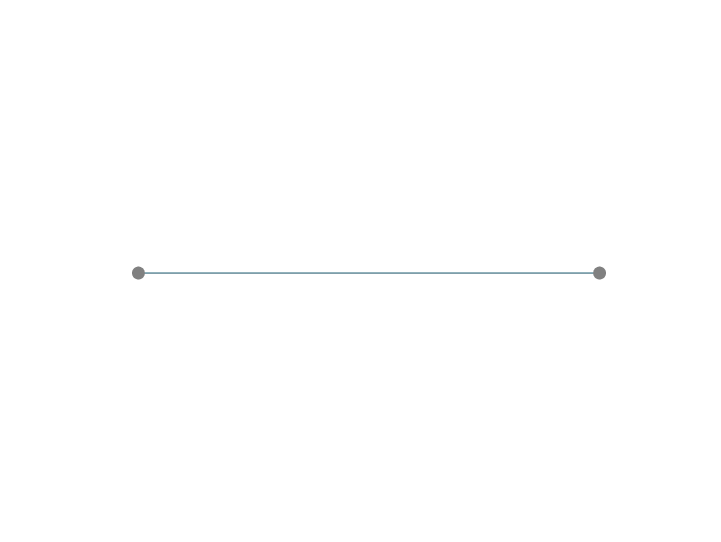}
\includegraphics[width=0.25\textwidth]{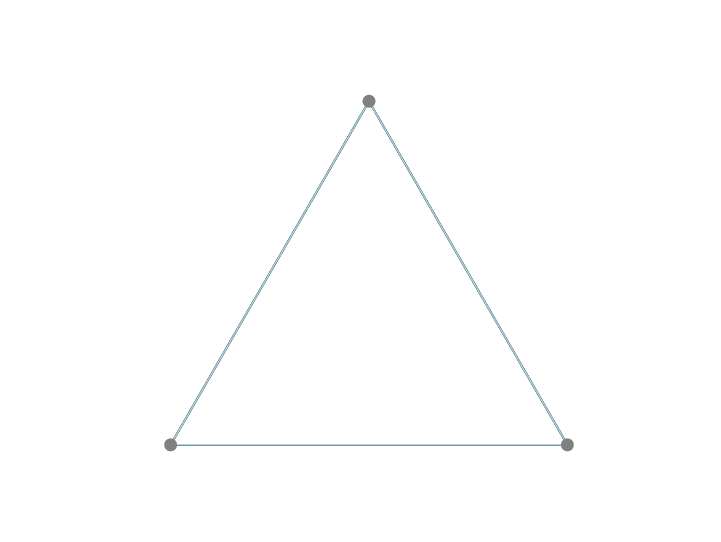}\\
\makebox[0pt][r]{}
\includegraphics[width=0.25\textwidth]{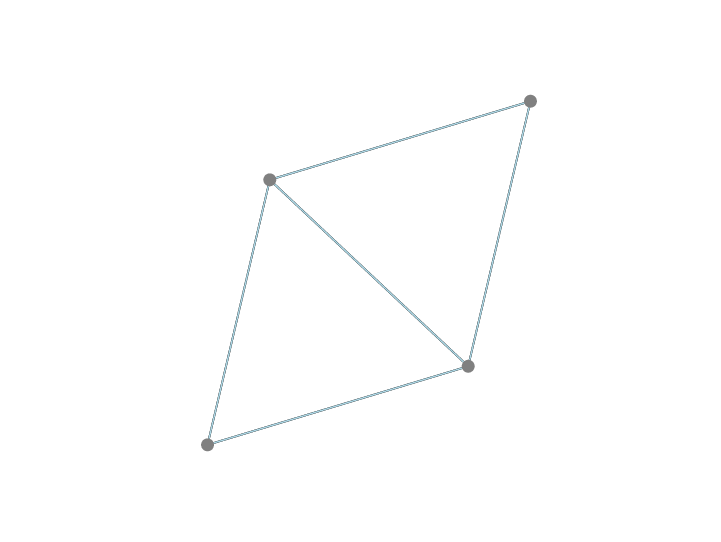}
\includegraphics[width=0.25\textwidth]{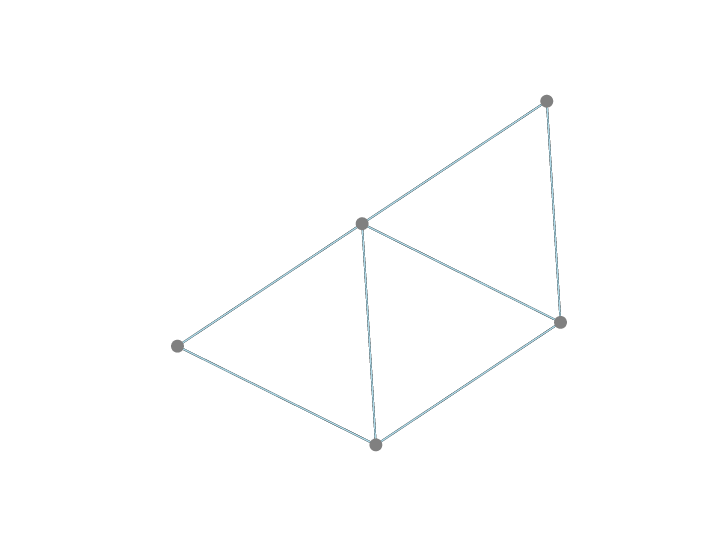}
\includegraphics[width=0.25\textwidth]{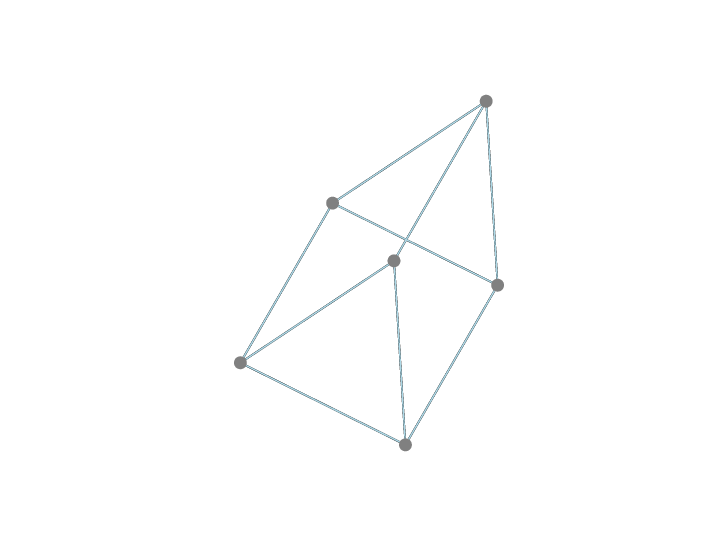}\\
\makebox[0pt][r]{}
\includegraphics[width=0.25\textwidth]{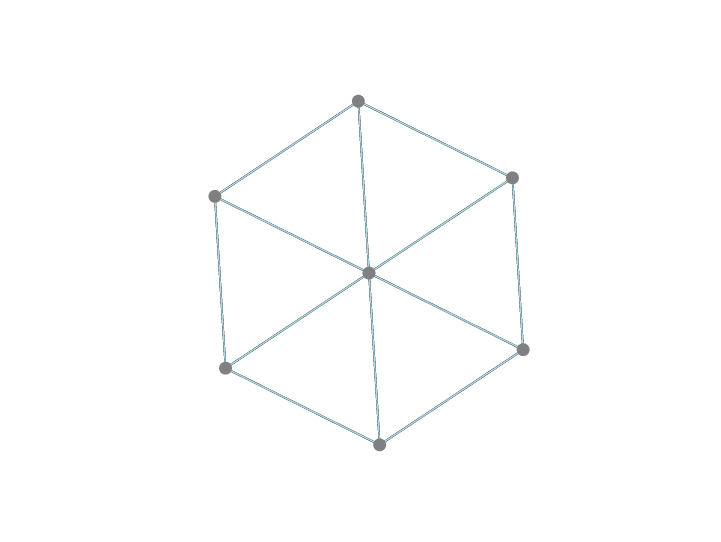}
\includegraphics[width=0.25\textwidth]{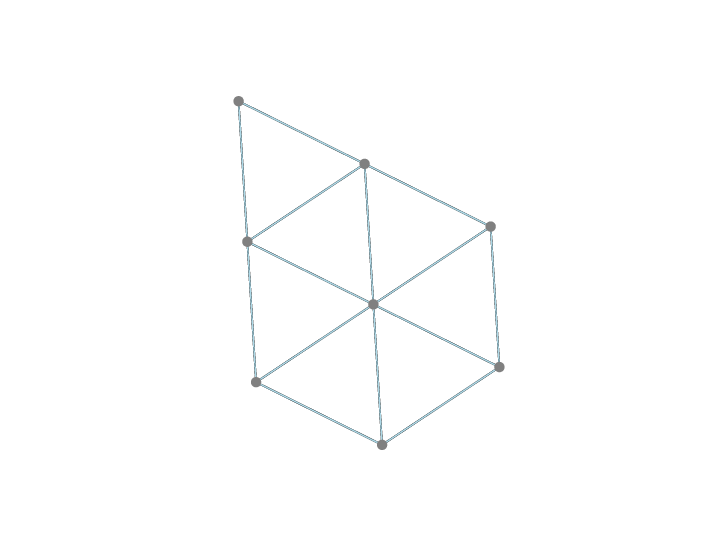}
\includegraphics[width=0.25\textwidth]{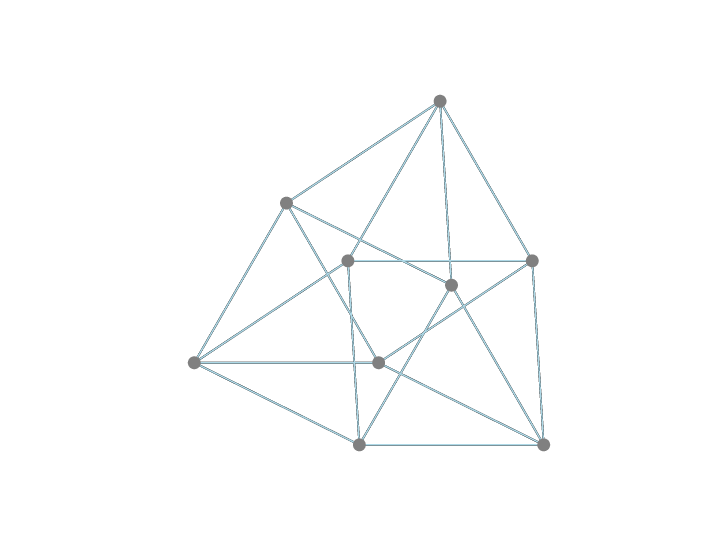}
\caption{Small vertex count maximal UDGs.}
\label{g1g9}
\end{figure}

\subsection{New Isomorphism Classes}
In addition to the graphs of 27--30 vertices found in \cite[Table 1]{agoston2022improved}, our algorithm generated previously unpublished graphs at those vertex counts with matching edge counts. For $n = 30$, a second isomorphism class with 93 edges was found. The new embeddings for 27--29 are shown in Figure~\ref{g27g29}, and the 30-vertex graphs are shown in Figure~\ref{g30}.

\begin{figure}
    \centering
    \begin{subfigure}{0.39\textwidth} % max width that fits
        \centering
        \includegraphics[width=\linewidth]{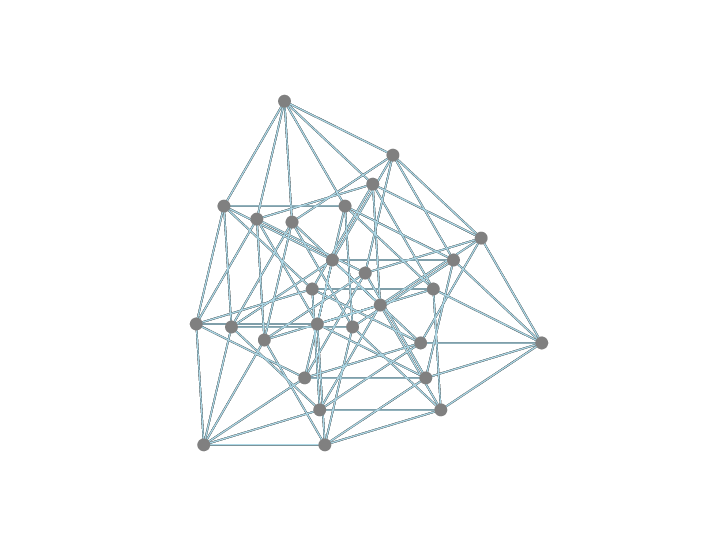}
        % \caption*{New 27-vertex graph of 81 edges}
    \end{subfigure}
    \hspace{-5em}
    \begin{subfigure}{0.39\textwidth} 
        \centering
        \includegraphics[width=\linewidth]{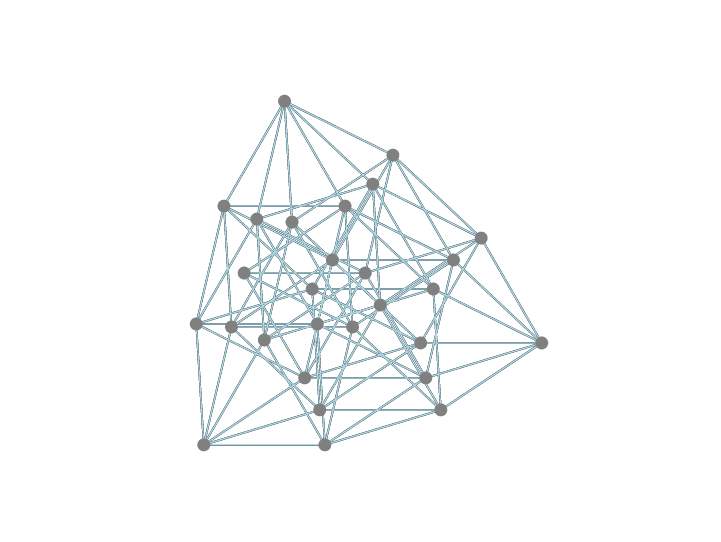}
        % \caption*{New 28-vertex graph of 85 edges}
    \end{subfigure}
    \hspace{-5em}
    \begin{subfigure}{0.39\textwidth} 
        \centering
        \includegraphics[width=\linewidth]{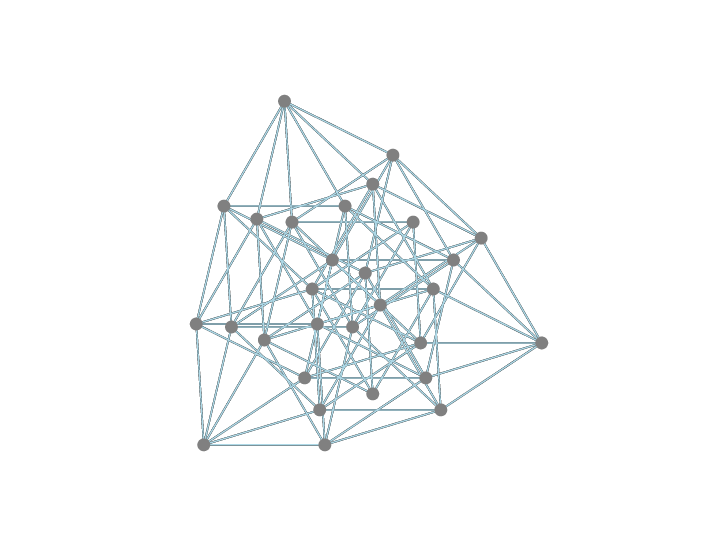}
        % \caption*{New 29-vertex graph of 89 edges}
    \end{subfigure}
    \caption{New 27, 28, and 29-vertex graphs, with edge counts 81, 85, 89, respectively.}
    \label{g27g29}
\end{figure}

\begin{figure}
    \centering
    \begin{subfigure}{0.54\textwidth} % Max width that fits
        \centering
        \includegraphics[width=\linewidth]{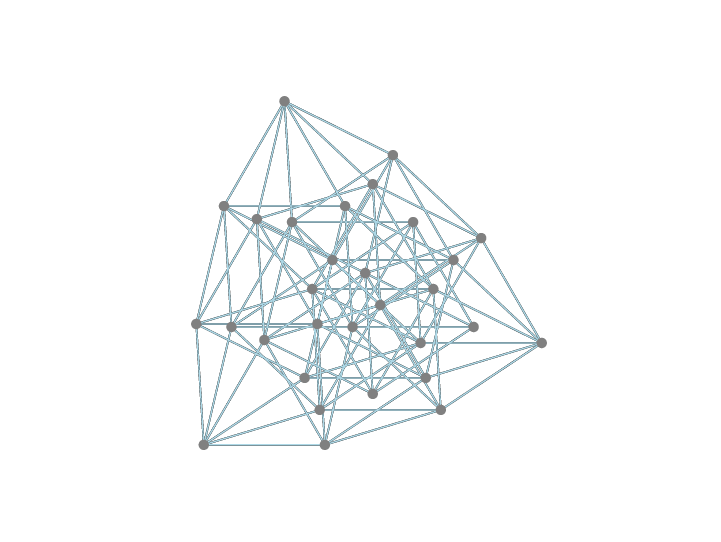}
    \end{subfigure}
    \hspace{-5em}
    \begin{subfigure}{0.54\textwidth}
        \centering
        \includegraphics[width=\linewidth]{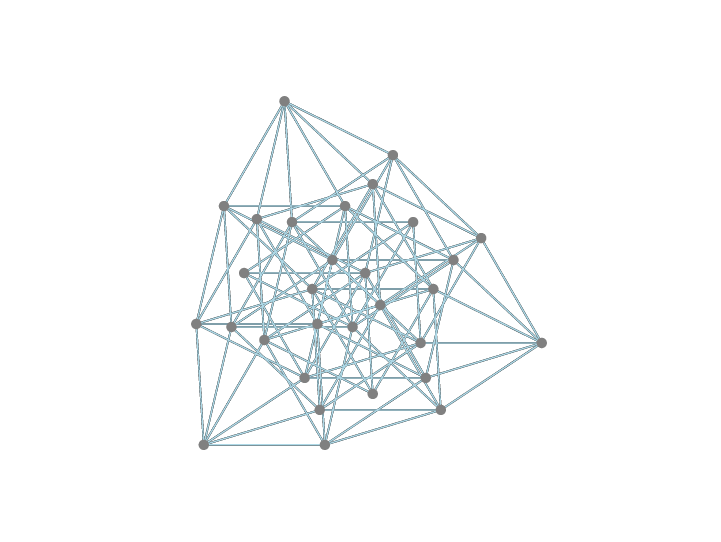}
    \end{subfigure}
    \caption{Two non-isomorphic 30-vertex graphs of 93 edges.}
    \label{g30}
\end{figure}

%\subsection{Forty-Nine Graph}
%One notable graph generated by our algorithm is the UDG of 49 vertices and 180 edges. %This is shown in Figure~\ref{g49}.
\begin{figure}
    \centering
    \includegraphics{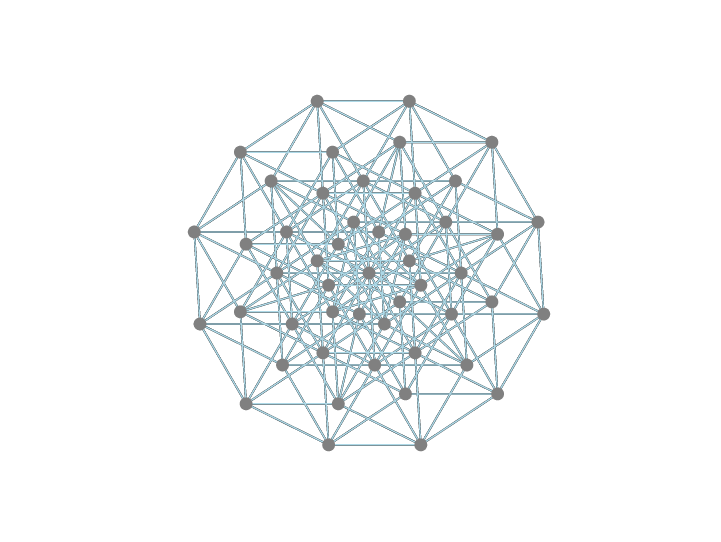}
    \caption{Graph of 49 vertices and 180 edges.}
    \label{g49}
    \end{figure}

\subsection{Minkowski sums}

Our search algorithm finds all the graphs documented in \cite[Table 1]{agoston2022improved}. This guarantees that it has found all known-optimal graphs for $n \leq 15$, and is a strong hint about its performance for $15 \leq n \leq 30$, where there are clear candidates, but their optimality is an open question. However, the question remains: are the graphs it finds for larger values of $n$ optimal? Surely they not always are, but below we give a completely informal heuristic argument for the optimality of some of the larger graphs the algorithm uncovers.

The Minkowski sum of two sets is defined as the pairwise sum of elements $A+B = \{a+b: a \in A, b \in B\}$. We will call this a \emph{disjoint} Minkowski sum, if $|A+B|=|A||B|$. An interesting property of many of the optimal small-vertex UDGs is that they are Minkowski sums, more specifically, they can be constructed from very small UDGs by repeated Minkowski summation. For example, the optimal 9-vertex UDG can be created by summing two unit triangles, and the conjectured-optimal 21-vertex UDG can be created as a disjoint Minkowski sum of a unit triangle and a 6-wheel, where the 6-wheel is itself a non-disjoint Minkowski sum of three edges.

There is only a single 49-vertex graph our algorithm uncovers as potentially optimal (Figure~\ref{g49}). It turns out to be a disjoint Minkowski sum of two 6-wheels, which makes it a non-disjoint Minkowski sum of 6 edges. This graph is identical to graph $G_{49}$ defined by \cite{Exoo2020}, a crucial ingredient to their proof to $\chi(\R^2) \geq 5$. 

Similarly, there is only a single 64-vertex graph that our algorithm uncovers as potentially optimal, and it turns out to be the disjoint Minkowski sum of 6 edges. In other words, it is a flattened 6-dimensional hypercube. 

Some more examples of this phenomenon: The algorithm uncovers potentially optimal 24-vertex and 28-vertex graphs that are non-disjoint Minkowski sums of 5 edges. The single 98-vertex graph our algorithm uncovers as potentially optimal is the Minkowski sum of 4 edges and a 7-vertex UDG. Two more examples are shown in Figures~\ref{minkowski1}~and~\ref{minkowski2}.

% https://docs.google.com/spreadsheets/d/15cySyrw0HulHbxpYhTCu3H6Z3vBFovjKgm2SUuiyo9E/edit#gid=0
It is notable that the fraction of Minkowski sums among the densest known UDGs is $44.2\%$, while the fraction of Minkowski sums among all the UDGs visited by the beam search is $5.6\%$. 
The algorithm does not have any intrinsic bias towards visiting Minkowski sums. The fact that the algorithm nevertheless uncovers these very regular structures can be considered weak indirect evidence that they are indeed optimal, at least among UDGs that are subsets of the Moser lattice.

\begin{figure}[h]
  \centering
  \begin{subfigure}[b]{0.3\linewidth}
    \centering
    \includegraphics[width=\linewidth]{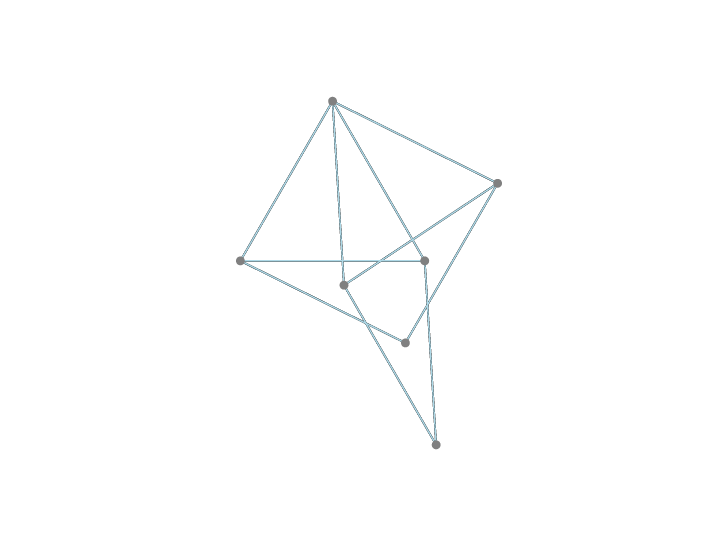}
  \end{subfigure}
    \begin{subfigure}[b]{0.02\linewidth}
    \centering
    \raisebox{3\height}{\huge{+}}
  \end{subfigure}
  \begin{subfigure}[b]{0.3\linewidth}
    \centering
    \includegraphics[width=\linewidth]{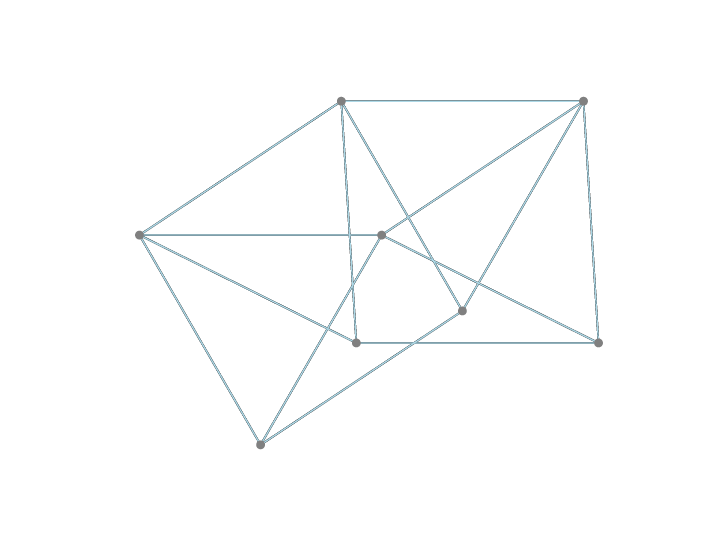}
  \end{subfigure}
    \begin{subfigure}[b]{0.02\linewidth}
    \centering
    \raisebox{5\height}{\huge{=}}
  \end{subfigure}
  \begin{subfigure}[b]{0.3\linewidth}
    \centering
    \includegraphics[width=\linewidth]{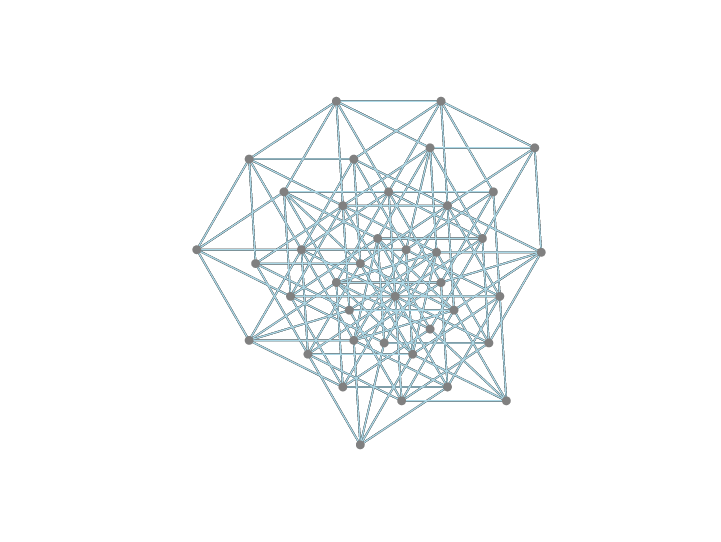}
  \end{subfigure}
  \caption{The Minkowski sum of a 7-vertex UDG and an 8-vertex UDG produces the densest-known 39-vertex UDG.}
  \label{minkowski1}
\end{figure}

\begin{figure}[h]
  \centering
  \begin{subfigure}[b]{0.3\linewidth}
    \centering
    \includegraphics[width=\linewidth]{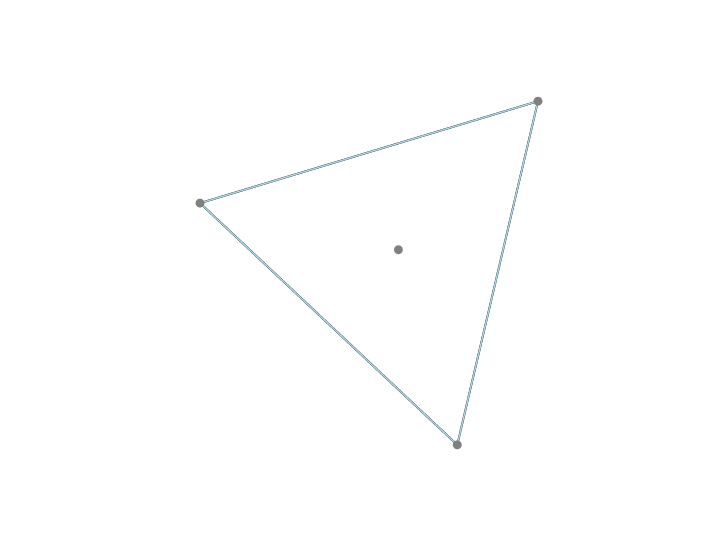}
  \end{subfigure}
    \begin{subfigure}[b]{0.02\linewidth}
    \centering
    \raisebox{3\height}{\huge{+}}
  \end{subfigure}
  \begin{subfigure}[b]{0.3\linewidth}
    \centering
    \includegraphics[width=\linewidth]{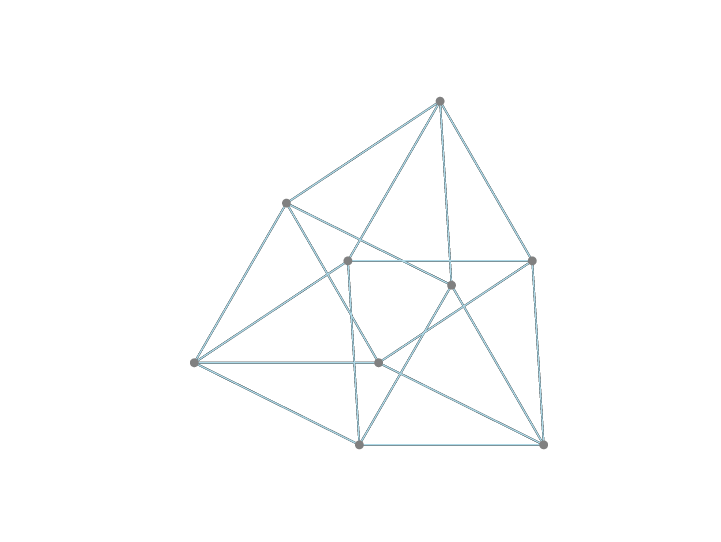}
  \end{subfigure}
    \begin{subfigure}[b]{0.02\linewidth}
    \centering
    \raisebox{5\height}{\huge{=}}
  \end{subfigure}
  \begin{subfigure}[b]{0.3\linewidth}
    \centering
    \includegraphics[width=\linewidth]{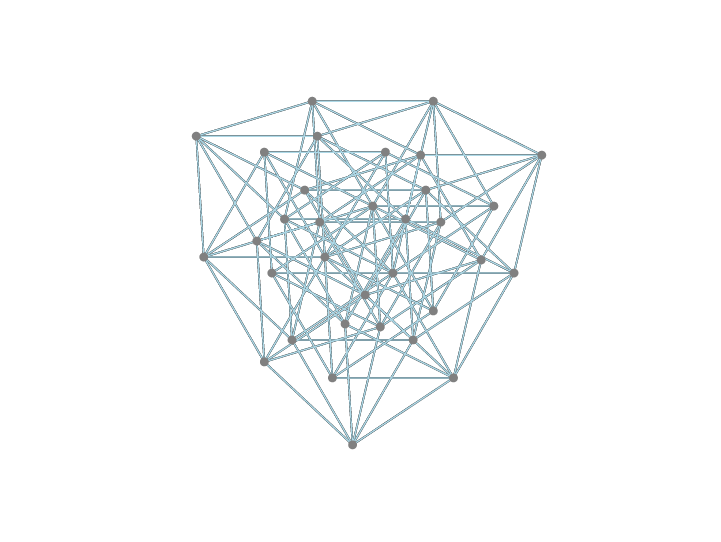}
  \end{subfigure}
  \caption{The Minkowski sum of a 4-vertex UDG---a triangle with an isolated center---and a 9-vertex UDG produces the densest-known 33-vertex UDG. The 9-vertex UDG is itself the Minkowski sum of two triangles.}
  \label{minkowski2}
\end{figure}

% Some more examples: 16 = 2 x 2 x 2 x 2. 40 = 2 x 2 x 2 x 2 x 3. 47 = 2 x 2 x 2 x 2 x 3.

\subsection{Finding 27}
Initial tests of the beam search matched the findings of Ágoston, Pálvölgyi \cite{agoston2022improved} up to 26 vertices, but failed to find known maximal graphs of 27-30 vertices. A separate computer search determined that finding the 27 graph required a graph somewhere in the parent chain with at least 5 fewer edges than is maximal. In a traditional beam search, this would require a beam large enough to include all graphs that are up to 4 edges sub-optimal. With diversity, it would require traversing every superior graph of that size a number of times equal to the difference in edge count. To solve this problem, we developed two changes to the algorithm. Firstly, an improved backward search, capable of traversing backward multiple levels instead of just one, and secondly the complete vectorization of the algorithm to allow larger beam widths. With these changes, not only did our algorithm find the previously known 27-30 graphs but also an additional family of such graphs. Having found all known maximal graphs only then did we feel confident in running the algorithm up to larger $n$.

\subsection{Vectorization Results}
As alluded to throughout the paper, vectorization increased the capabilities of the algorithm in terms of speed and beam width. In addition, several intermediary operations for tracking and database building purposes have been added to the vectorized code, which take up part of the shown runtime in Table~\ref{vecspeed}.
%no vectorization 
%30,10: 6.88 seconds\\
%30,100: 71.14 seconds\\
%vectorization 30,10: 4.05 seconds\\
%30,100: 9.405 seconds\\
%30,1000: 16.41 seconds\\
\begin{table}[h]
\begin{tabular}{|c|c|}
    \hline
    Beam width & Time to run algorithm to 30 vertices (s) \\
    \hline
    \hline
    \multicolumn{2}{|c|}{Without vectorization}\\
    \hline
    10 & 6.88 seconds\\
    \hline
    100 & 71.14 seconds \\
    \hline
    \multicolumn{2}{|c|}{With vectorization}\\
    \hline
    10 & 4.05 seconds\\
    \hline
    100 & 9.41 seconds\\
    \hline
    1000 & 16.41 seconds\\
    \hline
\end{tabular}
  \caption{Comparison of search times with and without vectorization.}
  \label{vecspeed}
\end{table}

Of course canonization and visitation mean that the same amount of computation also results in more novel graphs, further improving efficiency (though in ways less simply measured).

\subsection{General Results}
For vertex counts 1--100, we show the corresponding edge counts in Table~\ref{edgecounts}.

\begin{table}[h]
  \begin{minipage}{.22\textwidth}
    \centering
    \begin{tabular}{|c|c|c|}
      \hline
      V & E & I \\[0.5ex]
      \hline \hline
      1 & 0 & 1 \\
      \hline
      2 & 1 & 1 \\
      \hline
      3 & 3 & 1 \\
      \hline
      4 & 5 & 1 \\
      \hline
      5 & 7 & 1 \\
      \hline
      6 & 9 & 4\\
      \hline
      7 & 12 & 1\\
      \hline
      8 & 14 & 3\\
      \hline
      9 & 18 & 1 \\
      \hline
      10 & 20 & 1\\
      \hline
      11 & 23 & 2\\
      \hline
      12 & 27 & 1\\
      \hline
      13 & 30 & 1\\
      \hline
      14 & 33 & 2\\
      \hline
      15 & 37 & 1\\
      \hline
      16 & 41 & 1\\
      \hline
      17 & 43 & 6\\
      \hline
      18 & 46 & 16\\
      \hline
      19 & 50 & 3\\
      \hline
      20 & 54 & 1\\
      \hline
      21 & 57 & 5\\
      \hline
      22 & 60 & 35\\
      \hline
      23 & 64 & 10\\
      \hline
      24 & 68 & 7\\
      \hline
      25 & 72 & 3\\
      \hline
    \end{tabular}
  \end{minipage}
  \hfill
  \begin{minipage}{.22\textwidth}
    \centering
    \begin{tabular}{|c|c|c|}
      \hline
      V & E & I \\[0.5ex]
      \hline \hline
      26 & 76 & 2\\
      \hline
      27 & 81 & 1\\
      \hline
      28 & 85 & 2\\
      \hline
      29 & 89 & 1\\
      \hline
      30 & 93 & 2\\
      \hline
      31 & 97 & 2\\
      \hline
      32 & 101 & 4\\
      \hline
      33 & 105 & 9\\
      \hline
      34 & 109 & 20\\
      \hline
      35 & 114 & 1\\
      \hline
      36 & 119 & 1\\
      \hline
      37 & 123 & 5\\
      \hline
      38 & 128 & 1\\
      \hline
      39 & 132 & 6\\
      \hline
      40 & 137 & 1\\
      \hline
      41 & 141 & 3\\
      \hline
      42 & 146 & 1\\
      \hline
      43 & 150 & 5\\
      \hline
      44 & 155 & 1\\
      \hline
      45 & 160 & 1\\
      \hline
      46 & 164 & 4\\
      \hline
      47 & 169 & 2\\
      \hline
      48 & 174 & 2\\
      \hline
      49 & 180 & 1\\
      \hline
      50 & 183 & 5\\
      \hline
    \end{tabular}
  \end{minipage}
  \hfill
  \begin{minipage}{.22\textwidth}
    \centering
    \begin{tabular}{|c|c|c|}
      \hline
      V & E & I \\[0.5ex]
      \hline \hline
      51 & 188 & 2 \\
      \hline
      52 & 192 & 21\\
      \hline
      53 & 197 & 6\\
      \hline
      54 & 202 & 2\\
      \hline
      55 & 206 & 29\\
      \hline
      56 & 211 & 7\\
      \hline
      57 & 216 & 6\\
      \hline
      58 & 221 & 3\\
      \hline
      59 & 226 & 1\\
      \hline
      60 & 231 & 3\\
      \hline
      61 & 235 & 53\\
      \hline
      62 & 240 & 43\\
      \hline
      63 & 246 & 2\\
      \hline
      64 & 252 & 1\\
      \hline
      65 & 256 & 3\\
      \hline
      66 & 261 & 1\\
      \hline
      67 & 266 & 1\\
      \hline
      68 & 271 & 1\\
      \hline
      69 & 276 & 1\\
      \hline
      70 & 281 & 3\\
      \hline
      71 & 286 & 2\\
      \hline
      72 & 291 & 3\\
      \hline
      73 & 296 & 6\\
      \hline
      74 & 301 & 12\\
      \hline
      75 & 306 & 25\\
      \hline
    \end{tabular}
  \end{minipage}
    \hfill
  \begin{minipage}{.22\textwidth}
    \centering
    \begin{tabular}{|c|c|c|}
      \hline
      V & E & I \\[0.5ex]
      \hline \hline
      76 & 312 & 1\\
      \hline
      77 & 317 & 3\\
      \hline
      78 & 322 & 7\\
      \hline
      79 & 327 & 14\\
      \hline
      80 & 332 & 43\\
      \hline
      81 & 338 & 12\\
      \hline
      82 & 345 & 1\\
      \hline
      83 & 350 & 2\\
      \hline
      84 & 355 & 4\\
      \hline
      85 & 360 & 6\\
      \hline
      86 & 365 & 9\\
      \hline
      87 & 370 & 13\\
      \hline
      88 & 375 & 16\\
      \hline
      89 & 380 & 21\\
      \hline
      90 & 385 & 23\\
      \hline
      91 & 390 & 46\\
      \hline
      92 & 396 & 1\\
      \hline
      93 & 401 & 7\\
      \hline
      94 & 406 & 80\\
      \hline
      95 & 412 & 6\\
      \hline
      96 & 418 & 1\\
      \hline
      97 & 423 & 3\\
      \hline
      98 & 429 & 1\\
      \hline
      99 & 434 & 4\\
      \hline
      100 & 439 & 22\\
      \hline

    \end{tabular}
  \end{minipage}
  \caption{V is the number of vertices, E is the highest edge count found at V, and I is the number of found isomorphism classes of UDGs with V vertices and E edges.}
  \label{edgecounts}
\end{table}

%Future thoughts
\section{Further Work and Questions}
\subsection{Artificial Intelligence Training}
One of the outcomes of this research is a database of roughly sixty million graphs up to 100 vertices. This data provides a strong base for training artificial intelligence models to generate larger and larger graphs, with the eventual goal of studying the larger and larger known sets of graphs to generate sequences of new maximal graphs which improve upon previous lower bounds.
%\subsection{Applications of Visitation}
\subsection{The Moser Ring}
%TODO why the lattice is good
A natural extension of the Moser lattice that was shown in Theorem \ref{thm:18 unit vectors} to have degree 18 as a whole is the Moser ring, which may be a strong candidate for generating maximal graphs with vertex degrees greater than 18. The Moser ring, $M_R$ is defined as an extension of the Moser lattice including multiplication. $M_R$ has the usual addition and the multiplication operation defined as
\begin{eqnarray*}
    (a\cdot 1 + b\cdot\omega_1 + c\cdot\omega_3 + d\cdot\omega_1\omega_3) (a'\cdot 1 + b'\cdot\omega_1 + c'\cdot\omega_3 + d'\cdot\omega_1\omega_3),
\end{eqnarray*}
for all $a, b, c, d, a', b', c', d' \in \Z$. $M_R$ is an extension ring $\Z[\omega_1, \omega_3]$, where $\omega_1^2 = \omega_1^2 - 1$, $6\omega_3^2 = 10 \omega_3 -6$, and $\omega_3^2 = \frac{5}{3} \omega_3 -1$. 

Same as the Moser lattice, $M_R$ has an ordered base with 4 elements 
% has an ordered base with 24 elements 
\begin{eqnarray*}
    % B = \{ 1, \omega_3, \beta, \beta\omega_3, \alpha, \alpha\omega_3, \alpha\beta, \alpha\beta\omega_3, \alpha^2, \alpha^2\omega_3, \alpha^2\beta, \alpha^2\beta\omega_3,
    % \alpha^3, \alpha^3\omega_3, \alpha^3\beta, \\
    % \alpha^3\beta\omega_3, \alpha^4, \alpha^4\omega_3, \alpha^4\beta, \alpha^4\beta\omega_3, \alpha^5, \alpha^5\omega_3, \alpha^5\beta, \alpha^5\beta\omega_3 \}
    B = \{1, \omega_1, \omega_3, \omega_1\omega_3\},
\end{eqnarray*}
where $\omega_1 = \frac{1 + i\sqrt{3}}{2}$ and $\omega_3 = \frac{5 + i\sqrt{11}}{6}$. Similarly to the matrix form of UDGs with vertices in the Moser lattice, as given in Definition \ref{def:matrix form of udg}, every element $z\in M_R$ of the Moser ring can be written as $z=\frac{\sum_{b_i \in B} a_i b_i}{3^k}$ for some $(a_0, a_1, a_2, a_{3}, k)\in \mathbb{Z}^5$. Moreover, simplifying fractions, we get a unique such representation. Thus, for a UDG with $n$ vertices in the Moser ring, we could represent it by an $n \times 5$ integer matrix. All the beam search algorithms would be similar to the Moser lattice, except for the get-children function. We have a multiplication operator besides an addition operator for creating children of a UDG. Therefore, we potentially can get more children in the Moser ring than in the Moser lattice. 

Ideally, searching in the Moser ring could generate denser UDGs. Notably, this would not have the 18 edge degree limit of the Moser lattice. However, preliminary results indicated that transitioning to the Moser ring did not enhance the search outcomes. As a result, we search on the Moser lattice instead. Note that we have not implemented further optimizations on the Moser ring than what was done on the Moser lattice. It is possible that with additional optimizations, the beam search on the Moser ring could provide denser UDGs.

\subsection{Optimization}
Hyperparameter tuning could be used to optimize the beam width used for each size of graph, using specific graphs as benchmarks for the algorithm's progress. Further improvements could involve loss functions on the visitation values or new notions of diversity.

\section*{Acknowledgements} We would like to thank Helmut Ruhland for alternative lattice ideas and investigating lattices with many unit vectors \cite{ruhland2024families}.

This research was conducted under the auspices of the Budapest Semesters in Mathematics program's ``Research Opportunities'' initiative. D.~V.~and P.~Zs.~were supported by the Ministry of Innovation and Technology NRDI Office within the framework of the Artificial Intelligence National Laboratory (RRF-2.3.1-21-2022-00004). P.~Zs.~was partially supported by the Ministry of Innovation and Technology NRDI Office within the framework of the ELTE TKP 2021-NKTA-62 funding scheme.

%Large potential for optimization with a function taking in the visitation values, or another diversity notion altogether.
\bibliographystyle{plain}
\bibliography{ref.bib}
% https://arxiv.org/abs/2311.10069
% Moser and Moser Solution to problem 10
% cupy
% Zobrist hash https://research.cs.wisc.edu/techreports/1970/TR88.pdf
% asymptotic upper bound
% asymptotic lower bound
% Schade thesis

\vspace{1 cm}
\noindent

\medskip
\noindent
{\sc Peter Engel}
\smallskip

\noindent
{\em Cornell University Department of Mathematics, 301 Tower Rd, Ithaca, NY 14853}

\smallskip

\noindent
e-mail address: \texttt{pde23@cornell.edu}

\medskip
\noindent
{\sc Owen Hammond-Lee}
\smallskip

\noindent
{\em Georgia Institute of Technology School of Mathematics, 686 Cherry St NW, Atlanta, GA 30332}

\smallskip

\noindent
e-mail address: \texttt{ohammondlee3@gatech.edu}

\medskip
\noindent
{\sc Yiheng Su}
\smallskip

\noindent
{\em University of Wisconsin–Madison Department of Computer Science, 1210 W Dayton St, Madison, WI 53706}

\smallskip

\noindent
e-mail address: \texttt{su228@wisc.edu}

\medskip
\noindent
{\sc Dániel Varga}
\smallskip

\noindent
{\em HUN-REN Alfréd Rényi Institute of Mathematics, Reáltanoda u. 13-15, 1053,  Budapest, Hungary}

\smallskip

\noindent
e-mail address: \texttt{daniel@renyi.hu}

\medskip
\noindent
{\sc Pál Zsámboki} ({\em corresponding author})
\smallskip

\noindent
{\em HUN-REN Alfréd Rényi Institute of Mathematics, Reáltanoda u. 13-15, 1053, Budapest, Hungary\\ Institute of Mathematics, Faculty of Science, Eötvös Loránd University, Pázmány Péter sétány 1/C, 1117, Budapest, Hungary}

\smallskip

\noindent
e-mail address: \texttt{zsamboki.pal@renyi.hu}

\end{document}